\documentclass[12pt,reqno]{amsart}
\usepackage{amsmath,amssymb,amsthm,pinlabel}
\usepackage[pagewise]{lineno}
\usepackage[final, 
colorlinks=true,
pdftitle={Simultaneous construction of hyperbolic isometries},
pdfauthor={Matt Clay, Caglar Uyanik}]{hyperref}

\usepackage{tikz}
\usetikzlibrary{matrix,arrows,decorations.pathmorphing}

\theoremstyle{plain}
\newtheorem{theorem}{Theorem}
\newtheorem{lemma}[theorem]{Lemma}
\newtheorem{proposition}[theorem]{Proposition}

\newtheorem*{claim*}{Claim}

\theoremstyle{definition}
\newtheorem{definition}[theorem]{Definition}

\newtheorem{remark}[theorem]{Remark}

\numberwithin{equation}{section}
\numberwithin{theorem}{section}

\newcommand{\fakeenv}{} 

\newenvironment{restate}[2]  
{ 
 \renewcommand{\fakeenv}{#2} 
 \theoremstyle{plain} 
 \newtheorem*{\fakeenv}{#1~\ref{#2}} 
 \begin{\fakeenv}
}
{
 \end{\fakeenv}
}

\newcommand{\HH}{\mathbb{H}} 
\newcommand{\NN}{\mathbb{N}} 
\newcommand{\RR}{\mathbb{R}}

\newcommand{\ZZ}{\mathbb{Z}} 

\newcommand{\calA}{\mathcal{A}}
\newcommand{\calB}{\mathcal{B}}

\newcommand{\calF}{\mathcal{F}}

\newcommand{\calH}{\mathcal{H}}

\newcommand{\abs}[1]{\left\lvert {#1} \right\rvert} 
\newcommand{\I}[1]{\langle #1 \rangle}
\newcommand{\bd}{\partial}
\newcommand{\from}{\colon\thinspace}

\newcommand{\GP}[2]{\left( #1 \, . \, #2 \right)}

\newcommand{\fF}{\mathcal{FF}}

\newcommand{\param}%
	{{\mathchoice{\mkern1mu\mbox{\raise2.2pt\hbox{$\centerdot$}}\mkern1mu}%
	{\mkern1mu\mbox{\raise2.2pt\hbox{$\centerdot$}}\mkern1mu}%
	{\mkern1.5mu\centerdot\mkern1.5mu}{\mkern1.5mu\centerdot\mkern1.5mu}}}

\DeclareMathOperator{\Aut}{Aut}

\DeclareMathOperator{\IA}{IA}

\DeclareMathOperator{\Isom}{Isom}

\DeclareMathOperator{\Out}{Out}

\DeclareMathOperator{\rank}{rk}
\DeclareMathOperator{\rerank}{\underline{rk}}

\begin{document}



\title[Simultaneous construction of hyperbolic isometries]{Simultaneous construction of hyperbolic isometries}

\author[M.~Clay]{Matt Clay}
\address{Dept.\ of Mathematical Sciences \\
  University of Arkansas\\
  Fayetteville, AR 72701}
\email{\href{mailto:mattclay@uark.edu}{mattclay@uark.edu}}

\author[C.~Uyanik]{Caglar Uyanik}
\address{Dept.\ of Mathematics \\
  Vanderbilt University\\
  Nashville, TN, 37240}
\email{\href{mailto:caglar.uyanik@vanderbilt.edu}{caglar.uyanik@vanderbilt.edu}}

\thanks{\tiny The first author is partially supported by the Simons Foundation (award number 316383). The second author is partially supported by the NSF grants of Ilya Kapovich (DMS-1405146) and Christopher J. Leininger (DMS-1510034) and gratefully acknowledges support from U.S. National Science Foundation grants DMS 1107452, 1107263, 1107367 ``RNMS: GEometric structures And Representation varieties'' (the GEAR Network)}

\begin{abstract}
Given isometric actions by a group $G$ on finitely many $\delta$--hyperbolic metric spaces, we provide a sufficient condition that guarantees the existence of a single element in $G$ that is hyperbolic for each action.  As an application we prove a conjecture of Handel and Mosher regarding relatively fully irreducible subgroups and elements in the outer automorphism group of a free group~\cite{HMIntro}.
\end{abstract}

\maketitle


\section{Introduction}

A \emph{$\delta$--hyperbolic space} is a geodesic metric space where geodesic triangles are \emph{$\delta$--slim}: the $\delta$--neighborhood of any two sides of a geodesic triangle contains the third side.  Such spaces were introduced by Gromov in~\cite{col:Gromov87} as a coarse notion of negative curvature for geodesic metric spaces and since then have evolved into an indispensable tool in geometric group theory.  

There is a classification of isometries of $\delta$--hyperbolic metric spaces analogous to the classification of isometries of hyperbolic space $\HH^{n}$ into elliptic, hyperbolic and parabolic.  Of these, hyperbolic isometries have the best dynamical properties and are often the most desired.  For example, typically they can be used to produce free subgroups in a group acting on a $\delta$--hyperbolic space~\cite[5.3B]{col:Gromov87}, see also~\cite[III.$\Gamma$.3.20]{bk:BH99}.  Another application is to show that a certain element does not have fixed points in its action on some set.  Indeed, if the set naturally sits inside of a $\delta$--hyperbolic metric space and the given element acts as a hyperbolic isometry then it has no fixed points (in a strong sense).  This strategy has been successfully employed for the curve complex of a surface and for the free factor complex of a free group by several authors~\cite{ar:CLM12,ar:CP12-1,DTHyp,ar:Fujiwara15,un:Gultepe, Hshort,ar:Mangahas13,ar:Taylor14}.   

We consider the situation of a group acting on finitely many $\delta$--hyperbolic spaces and produce a sufficient condition that guarantees the existence of a single element in the group that is a hyperbolic isometry for each of the spaces.  Of course, a necessary condition is that for each of the spaces there is some element of the group that is a hyperbolic isometry.  Thus we are concerned with when we may reverse the quantifiers: $\forall \exists \leadsto \exists\forall$.  Our main result is the following theorem.    

\begin{restate}{Theorem}{th:constructing hyperbolic actions}
Suppose that $\{ X_{i} \}_{i = 1,\ldots, n}$ is a collection of $\delta$--hyperbolic spaces, $G$ is a group and for each $i = 1,\ldots, n$ there is a homomorphism $\rho_{i} \from G \to \Isom(X_{i})$ such that:
\begin{enumerate}
\item there is an element $f_{i} \in G$ such that $\rho_{i}(f_{i})$ is hyperbolic; and
\item for each $g \in G$, either $\rho_{i}(g)$ has a periodic orbit or is hyperbolic.
\end{enumerate}
Then there is an $f \in G$ such that $\rho_{i}(f)$ is hyperbolic for all $i = 1,\ldots,n$.
\end{restate}

\begin{remark} After the completion of this paper we have been alerted that Theorem \ref{th:constructing hyperbolic actions} should follow from random walk techniques developed in \cite{BH10} and \cite{MT}. Here we provide an elementary and constructive proof. 
\end{remark}

Essentially, we assume that there are no parabolic isometries and that elliptic isometries are relatively tame.  

As an application of our main theorem we prove a conjecture of Handel and Mosher which exactly involves the same type of quantifier reversing: $\forall \exists \leadsto \exists\forall$.  Consider a finitely generated subgroup $\calH<\IA_{N}(\ZZ/3)<\Out(F_N)$ and a maximal $\calH$--invariant filtration of $F_{N}$, the free group of rank $N$, by free factor systems
\[\emptyset = \calF_{0} \sqsubset \calF_{1} \sqsubset \cdots \sqsubset \calF_{m} = \{[ F_{N}] \}\]
(see Section \ref{sec:application}). Handel and Mosher prove that for each multi-edge extension $\calF_{i-1} \sqsubset \calF_{i}$ there exists some $\varphi_{i} \in \calH$ that is irreducible with respect to $\calF_{i-1} \sqsubset \calF_{i}$~\cite[Theorem~D]{HMIntro}.  They conjecture that there exists a single $\varphi \in \calH$ that is irreducible with respect to each multi-edge extension $\calF_{i-1} \sqsubset \calF_{i}$.  We show that this is indeed the case.

\begin{restate}{Theorem}{th:application}
For each finitely generated subgroup $\calH < \IA_{N}(\ZZ/3)<\Out(F_N)$ and each maximal $\calH$--invariant filtration by free factor systems $\emptyset = \calF_{0} \sqsubset \calF_{1} \sqsubset \cdots \sqsubset \calF_{m} = \{[ F_{N}] \}$, there is an element $\varphi \in \calH$ such that for each $i = 1,\ldots,m$ such that $\calF_{i-1} \sqsubset \calF_{i}$ is a multi-edge extension, $\varphi$ is irreducible with respect to $\calF_{i-1} \sqsubset \calF_{i}$.
\end{restate}

Our paper is organized as follows.  Section~\ref{sec:spaces} contains background on $\delta$--hyperbolic spaces and their isometries.  In Section~\ref{sec:actions} we generalize a construction of the first author and Pettet from~\cite{ar:CP12-1} that is useful to constructing hyperbolic isometries.  This result is Theorem~\ref{th:uniform hyperbolic}.  We examine certain cases that will arise in the proof of the main theorem to see how to apply Theorem~\ref{th:uniform hyperbolic} in Section~\ref{sec:neighborhoods}.  The proof of Theorem~\ref{th:constructing hyperbolic actions} constitutes Section~\ref{sec:simultaneous}.  The application to $\Out(F_{N})$ appears in Section~\ref{sec:application}.

\subsection*{Acknowledgements.}  We would like to thank Lee Mosher and Camille Horbez for useful discussions. We are grateful to Camille Horbez for informing us about his work with Vincent Guirardel~\cite{un:GH}.  We thank the referee for a careful reading and for providing useful suggestions.  The second author thanks Ilya Kapovich and Chris Leininger for guidance and support.  


\section{Background on $\delta$--hyperbolic spaces}\label{sec:spaces}

In this section we recall basic notions and facts about $\delta$--hyperbolic spaces, their isometries and their boundaries. The reader familiar with these topics can safely skip this section with  the exception of Definition~\ref{def:independent}.  References for this section are \cite{col:AlonsoEtAl91}, \cite{bk:BH99} and \cite{col:KB02}.

\subsection{$\delta$--hyperbolic spaces}\label{subsec:spaces}

We recall the definition of a $\delta$--hyperbolic space given in the Introduction.

\begin{definition}\label{def:hyperbolic}
Let $(X,d)$ be a geodesic metric space.  A geodesic triangle with sides $\alpha$, $\beta$ and $\gamma$ is \emph{$\delta$--slim} if for each $x \in \alpha$, there is some $y \in \beta \cup \gamma$ such that $d(x,y) \leq \delta$.  The space $X$ is said to be \emph{$\delta$--hyperbolic} if every geodesic triangle is $\delta$--slim.  
\end{definition}

There are several equivalent definitions that we will use in the sequel.  The first of these is insize.  Let $\Delta$ be the geodesic triangle with vertices $x$, $y$ and $z$ and sides $\alpha$ from $y$ to $z$, $\beta$ from $z$ to $x$ and $\gamma$ from $x$ to $y$.  There exist unique points $\hat{\alpha} \in \alpha$, $\hat{\beta} \in \beta$ and $\hat{\gamma} \in \gamma$, called the \emph{internal points} of $\Delta$, such that:
\begin{equation*}
d(x,\hat{\beta}) = d(x,\hat{\gamma}), \,
d(y,\hat{\gamma}) = d(y,\hat{\alpha}) \text{ and }
d(z,\hat{\alpha}) = d(z,\hat{\beta}).
\end{equation*}
The \emph{insize} of $\Delta$ is the diameter of the set $\{ \hat{\alpha},\hat{\beta},\hat{\gamma}\}$.  
  
Another notion makes use of the so-called \emph{Gromov product}:
\begin{equation}\label{eq:gp}
\GP{x}{y}_{w} = \frac{1}{2}(d(x,w) + d(w,y) - d(x,y)).
\end{equation}
The Gromov product is said to be \emph{$\delta$--hyperbolic \textup{(}with respect to $w \in X$\textup{)}} if for all $x,y,z \in X$:
\[ \GP{x}{z}_{w} \geq \min\left\{ \GP{x}{y}_{w}, \GP{y}{z}_{w} \right\} - \delta. \]
 
\begin{proposition}[{\cite[Proposition~2.1]{col:AlonsoEtAl91}, \cite[III.H.1.17 and III.H.1.22]{bk:BH99}}]\label{prop:delta-hyp}
The following are equivalent for a geodesic metric space $X$:
\begin{enumerate}
\item There is a $\delta_{1} \geq 0$ such that every geodesic triangle in $X$ is $\delta_{1}$--slim, i.e., $X$ is $\delta_{1}$--hyperbolic.
\item There is a $\delta_{2} \geq 0$ such that every geodesic triangle in $X$ has insize at most $\delta_{2}$.
\item There is a $\delta_{3} \geq 0$ such that for some 
\textup{(}equivalently any\textup{)} $w \in X$, the Gromov product is $\delta_{3}$--hyperbolic.
\end{enumerate}
\end{proposition}
Henceforth, when we say $X$ is a $\delta$--hyperbolic space we assume that $\delta$ is large enough to satisfy each of the above conditions.
 
\subsection{Boundaries}\label{subsec:boundaries}

There is a useful notion of a boundary for a $\delta$--hyperbolic space that plays the role of the ``sphere at infinity'' for $\HH^{n}$.  This space is defined using equivalence classes of certain sequences of points in $X$ and the Gromov product.  Fix a basepoint $w \in X$.  

\begin{definition}\label{def:boundary}
We say a sequence $(x_{n}) \subseteq X$ \emph{converges to infinity} if $\GP{x_{i}}{x_{j}}_{w} \to \infty$ as $i,j \to \infty$.  Two such sequences $(x_{n})$, $(y_{n})$ are equivalent if $\GP{x_{i}}{y_{j}}_{w} \to \infty$ as $i,j \to \infty$.  The \emph{boundary of $X$}, denoted $\partial X$, is the set of equivalence classes of sequences $(x_{n}) \subseteq X$ that converge to infinity. 
\end{definition}

One can show that the notion of ``converges to infinity'' and the subsequent equivalence relation do not depend on the choice of basepoint $w \in X$~\cite{col:KB02}.
The definition of the Gromov product in~\eqref{eq:gp} extends to boundary points $\hat{x}, \hat{y} \in \partial X$ by:
\[ \GP{\hat{x}}{\hat{y}}_{w} = \inf \{ \liminf_{n} \GP{x_{n}}{y_{n}}_{w} \} \]
where the infimum is over sequences $(x_{n}) \in \hat{x}$, $(y_{n}) \in \hat{y}$.  If $y \in X$ then we set:
\[ \GP{\hat{x}}{y}_{w} = \inf \{ \liminf_{n} \GP{x_{n}}{y}_{w} \} \]
where the infimum is over sequences $(x_{n}) \in \hat{x}$.  For $x \in X$, the Gromov product $\GP{x}{\hat{y}}_{w}$ is defined analogously.  Let $\overline{X} = X \cup \partial X$.

We will make use of the following properties of the Gromov product on $\overline{X}$.

\begin{proposition}[{\cite[Lemma~4.6]{col:AlonsoEtAl91}, \cite[III.H.3.17]{bk:BH99}}]\label{prop:extended gp}
Let $X$ be a $\delta$--hyperbolic space.
\begin{enumerate}
\item If $x,y \in \overline{X}$ then $\GP{x}{y}_{w} = \infty \iff x = y \in \partial X$. 
\item If $\hat{x} \in \partial X$ and $(x_{n}) \subseteq X$ then $\GP{\hat{x}}{x_{n}}_{w} \to \infty$ as $n \to \infty \iff (x_{n}) \in \hat{x}$.
\item If $\hat{x},\hat{y} \in \partial X$ and $(x_{n}) \in \hat{x}$, $(y_{n}) \in \hat{y}$ then:
\[ \GP{\hat{x}}{\hat{y}}_{w} \leq \liminf_{n} \GP{x_{n}}{y_{n}}_{w} \leq \GP{\hat{x}}{\hat{y}}_{w} - 2\delta. \]
\item If $x,y,z \in \overline{X}$ then:
\[ \GP{x}{z}_{w} \geq \min\left\{ \GP{x}{y}_{w}, \GP{y}{z}_{w} \right\} - \delta. \]
\end{enumerate}
\end{proposition}

\begin{proposition}[{\cite[Proposition~4.8]{col:AlonsoEtAl91}}]\label{prop:basis}
The following collection of subsets of $\overline{X}$ forms a basis for a topology:
\begin{enumerate}
\item $B(x,r) = \{ y \in X \mid d(x,y) < r \}$, for each $x \in X$ and $r > 0$; and
\item $N(\hat{x},k) = \{ y \in \overline{X} \mid \GP{\hat{x}}{y}_{w} > k \}$ for each $\hat{x} \in \partial X$ and $k > 0$.
\end{enumerate}
\end{proposition}

\subsection{Isometries}\label{subsec:isometries}
 
As mentioned in the Introduction, there is a classification of isometries of a $\delta$--hyperbolic space $X$ into elliptic, parabolic and hyperbolic~\cite[8.1.B]{col:Gromov87}.  We will not make use of parabolic isometries and so do not give the definition here.

\begin{definition}\label{def:isometries}
An isometry $f \in \Isom(X)$ is \emph{elliptic} if for any $x \in X$, the set $\{ f^{n}x \mid n \in \ZZ \}$ has bounded diameter.  

An isometry $f \in \Isom(X)$ is \emph{hyperbolic} if for any $x \in X$ there is a $t > 0$ such that $t \abs{m-n} \leq d(f^{m}x,f^{n}x)$ for all $m,n \in \ZZ$.  In this case, one can show, the sequence $(f^{n}x) \subseteq X$ converges to infinity and the equivalence class it defines in $\partial X$ is independent of $x \in X$.  This point in $\partial X$ is called the \emph{attracting fixed point} of $f$.  The \emph{repelling fixed point} of $f$ is the attracting fixed point of $f^{-1}$ and is represented by the sequence $(f^{-n}x) \subseteq X$.
\end{definition}
 
The action of a hyperbolic isometry $f \in \Isom(X)$ on $\overline{X}$ has ``North-South dynamics.''

\begin{proposition}[{\cite[8.1.G]{col:Gromov87}}]\label{prop:ns} 
Suppose that $f \in \Isom(X)$ is a hyperbolic isometry and that $U_{+}, U_{-} \subset \overline{X}$ are disjoint neighborhoods of the attracting and repelling fixed points of $f$ respectively.  There exists an $N \geq 1$ such that for $n \geq N$:
\begin{equation*}
 f^{n}(\overline{X} - U_{-}) \subseteq U_{+} \text{ and }
 f^{-n}(\overline{X} - U_{+}) \subseteq U_{-}.
\end{equation*}
\end{proposition}

We will make use of the following definition.

\begin{definition}\label{def:independent}
Suppose $X$ is a $\delta$--hyperbolic space and $f,g \in \Isom(X)$ are hyperbolic isometries.  Let $A_{+}$, $A_{-}$ be the attracting and repelling fixed points of $f$ in $\partial X$ and let $B_{+}$, $B_{-}$ be the attracting and repelling fixed points of $g$ in $\partial X$.  We say $f$ and $g$ are \emph{independent} if:
\[ \{A_{+},A_{-}\} \cap \{B_{+},B_{-}\} = \emptyset. \]
Hyperbolic isometries that are not independent are said to be \emph{dependent}.
\end{definition}


\section{A recipe for hyperbolic isometries}\label{sec:actions}

In this section we prove the principal tool used in the proof of the main result of this article, producing a single element in the given group that is hyperbolic for each action.  The idea is to start with elements $f$ and $g$ that are hyperbolic for different actions and then combine them into a single element $f^{a}g^{b}$ that is hyperbolic for both actions.  A theorem of the first author and Pettet shows that if $g$ does not send the attracting fixed point of $f$ to the repelling fixed point, then $f^{a}g$ is hyperbolic in the first action for large enough $a$.  We can reverse the roles to get that $fg^{b}$ is hyperbolic in the second action for large enough $b$.  In order to simultaneously work with powers for both $f$ and $g$, we need a uniform version of this result.  That is the content of the next theorem, which generalizes Theorem~4.1 in~\cite{ar:CP12-1}.

\begin{theorem}\label{th:uniform hyperbolic}
Suppose $X$ is a $\delta$--hyperbolic space and $f \in \Isom(X)$ is a hyperbolic isometry with attracting and repelling fixed points $A_{+}$ and $A_{-}$ respectively.  Fix disjoint neighborhoods $U_{+}$ and $U_{-}$ in $\overline{X}$ for $A_{+}$ and $A_{-}$ respectively.  Then there is an $M \geq 1$ such that if $m \geq M$ and $g \in \Isom(X)$  then $f^{m}g$ is a hyperbolic isometry whenever $g U_{+} \cap U_{-} = \emptyset$.   
\end{theorem}

The proof follows along the lines of Theorem~4.1 in \cite{ar:CP12-1}.  In the following two lemmas we assume the hypotheses of Theorem~\ref{th:uniform hyperbolic}.  The first lemma is obvious in the hypothesis of Theorem~4.1 in \cite{ar:CP12-1} but requires a proof in this setting.

\begin{lemma}\label{lem:uniform hyperbolic 1}
Given a point $x \in U_{+} \cap X$ there are constants $t >0 $ and $C \geq 0$ such that if $g \in \Isom(X)$ is such that $g U_{+} \cap U_{-} = \emptyset$ then $d(x,f^{m}gx) \geq mt - C$ for all $m \geq 0$.
\end{lemma}

\begin{proof}
Let $A = \{ f^{n}x | n \in \ZZ\}$ and for $z \in X$ let \[d_{z} = \inf\{ d(x',z) \mid x' \in A\}.\]  As $f$ is a hyperbolic isometry, there is a constant $\tau \geq 1$ such that:
\[ \frac{1}{\tau}\abs{m-n} \leq d(f^{m}x,f^{n}x) \leq \tau\abs{m-n}. \]  This shows that for any $z \in X$ the set $\pi_{z} = \{ x' \in A \mid d(x',z) = d_{z} \}$ is nonempty and finite.   

\medskip \noindent {\bf Claim 1:} {\it There is a constant $D \geq 0$ such that for any $z \in X$ and $x_{z} \in \pi_{z}$:}
\[ d(x,z) \geq d(x,x_{z}) + d(x_{z},z) - D. \]

\begin{proof}[Proof of Claim 1]  
Fix a point $x_{z} \in \pi_{z}$ and geodesics $\alpha$ from $x_{z}$ to $x$, $\beta$ from $z$ to $x_{z}$ and $\gamma$ from $z$ to $x$.  Let $\Delta$ be the geodesic triangle formed with these segments and $\hat{\alpha} \in \alpha$, $\hat{\beta} \in \beta$ and $\hat{\gamma} \in \gamma$ be the internal points of $\Delta$.  These points satisfy the equalities:
\begin{align*}
d(z,\hat{\beta}) = d(z,\hat{\gamma}) &= a \\
d(x,\hat{\gamma}) = d(x,\hat{\alpha}) & = b \\
d(x_{z},\hat{\alpha}) = d(x_{z},\hat{\beta}) & = c 
\end{align*}
As insize of geodesic triangles is bounded by $\delta$ in a $\delta$--hyperbolic space, we have that $d(\hat{\alpha},\hat{\beta}), d(\hat{\beta},\hat{\gamma}), d(\hat{\gamma},\hat{\alpha}) \leq \delta$.  By the Morse lemma~\cite[III.H.1.7]{bk:BH99}, there is a constant $R$, only depending on $\tau$ and $\delta$, and a point $y \in A$ such that $d(\hat{\alpha},y) \leq R$.  Thus we have that:
\[ d(z,y) \leq d(z,\hat{\beta}) + d(\hat{\beta},\hat{\alpha}) + d(\hat{\alpha},y) \leq a + \delta + R.\]
As $x_{z} \in \pi_{z}$ we have:
\[a + c = d(x_{z},z) \leq d(z,y) \leq a + \delta + R\] 
and so $c \leq \delta + R$.  Letting $D = 2\delta + 2R$ we compute: 
\begin{align*}
d(x,z) & = a + b \\
&= (b+c) + (a+c) - 2c \\
&\geq d(x,x_{z}) + d(x_{z},z) - D.\qedhere
\end{align*}
\end{proof}

\medskip \noindent {\bf Claim 2:} {\it There is a constant $M_{0} \in \ZZ$ such that if $z \notin U_{-}$ and $f^{m}x \in \pi_{z}$ then $m \geq M_{0}$.}

\begin{proof}[Proof of Claim 2]
Let $x_{z} = f^{m}x \in \pi_{z}$ and without loss of generality assume that $m \leq 0$.  Using the constant $D$ from Claim 1 we have:
\begin{align*}
\GP{x_{z}}{z}_{x} & = \frac{1}{2}\left(d(x,x_{z}) + d(x,z) - d(x_{z},z)\right) \\
& \geq d(x,x_{z}) - D/2.
\end{align*}
Suppose that $i \leq m$ and let $\alpha$ be a geodesic from $f^{i}x$ to $x$.  The Morse lemma implies that there is an $y \in \alpha$ such that $d(x_{z},y) \leq R$.  Therefore:
\begin{align*}
d(x,x_{z}) + d(x_{z},f^{i}x) &\leq d(x,y) + d(y,f^{i}x) + 2R \\
&= d(x,f^{i}x) + 2R.
\end{align*} 
Hence for such $i$ we have:
\begin{align*}
\GP{x_{z}}{f^{i}x}_{x} &= \frac{1}{2}\left(d(x,x_{z}) + d(x,f^{i}x) - d(x_{z},f^{i}x)\right) \\
& \geq d(x,x_{z}) - R.
\end{align*} 
This shows that $\GP{x_{z}}{A_{-}}_{x} \geq d(x,x_{z}) - R - 2\delta$ and so for $K = \max\{ D/2,R + 2\delta \}$ we have:
\[ \GP{z}{A_{-}}_{x} \geq \min \left\{ \GP{x_{z}}{z}_{x},\GP{x_{z}}{A_{-}}_{x} \right\} - \delta  \geq d(x,x_{z}) - K - \delta \]
As $z \notin U_{-}$, the Gromov product $\GP{z}{A_{-}}_{x}$ is bounded independently of $z$ and hence $d(x,x_{z})$ is also bounded.
\end{proof}

Now we will finish the proof of the lemma.  Fix a point $x_{g} \in \pi_{gx}$.  Clearly we have $f^{m}x_{g} \in \pi_{f^{m}gx}$ for $m \geq 0$.  As $gx \notin U_{-}$, by Claim 2 we have $x_{g} = f^{M_{0}+n}x$ for some $n \geq 0$ and therefore:
\begin{align*}
d(x,f^{m}x_{g}) = d(x,f^{M_{0}+n+m}x) &\geq d(x,f^{m+n}x) - d(x,f^{M_{0}}x) \\
&\geq \frac{1}{\tau}m - \tau \abs{M_{0}}.
\end{align*}
As $f^{m}x_{g} \in \pi_{f^{m}gx}$, Claim 1 implies:
\begin{align*}
d(x,f^{m}gx) &\geq d(x,f^{m}x_{g}) + d(f^{m}x_{g},f^{m}gx) - D \\
& \geq \frac{1}{\tau}m - (\tau\abs{M_{0}} + D).
\end{align*}  
Since the constants $\tau$, $D$ and $M_{0}$ only depend on $f$, $x$ and the open neighborhoods $U_{+}$ and $U_{-}$, the lemma is proven.
\end{proof}

The next lemma replaces Lemma~4.3 in~\cite{ar:CP12-1} and its proof is a small modification of the proof there.

\begin{lemma}\label{lem:uniform hyperbolic 2}
Fix $x \in X \cap U_{+}$ and for $m \geq 0$ let $\alpha_{m}$ be a geodesic connecting $x$ to $f^{m}gx$.  Then there is an $\epsilon \geq 0$ and $M_{1} \geq 0$ such that for $m \geq M_{1}$ the concatenation of the geodesics $\alpha_{m} \cdot f^{m}g\alpha_{m}$ is a $(1,\epsilon)$-quasi-geodesic.
\end{lemma}

\begin{proof}
Let $d_{m} = d(x,f^{m}gx)$.
 
As $gU_{+} \cap U_{-} = \emptyset$ we have $U_{+} \cap g^{-1}U_{-} = \emptyset$ and so the Gromov product $\GP{g^{-1}f^{-m}x}{f^{m}x}_{x}$ is bounded independent of $g$ and $m \geq M_{1}$ for some constant $M_{1}$.  Indeed, by Proposition~\ref{prop:basis} there is a $k \geq 0$ such that $N(A_{+},k) \subseteq U_{+}$ and $M_{1} \geq 0$ such that $f^{-m}x \in U_{-}$ and $f^{m}x \in N(A_+,k + 2\delta)$ for $m \geq M_{1}$.   Hence $\GP{A_{+}}{g^{-1}f^{-m}x}_{x}  \leq k$ and so $\GP{g^{-1}f^{-m}x}{f^{m}x}_{x} \leq k + \delta$ as:
\begin{equation*}
\min\{\GP{A_{+}}{f^{m}x}_{x}, \GP{g^{-1}f^{-m}x}{f^{m}x}_{x}\} - \delta \leq  \GP{A_{+}}{g^{-1}f^{-m}x}_{x}  \leq k
\end{equation*}
for $m \geq M_{1}$.
 
By making $M_{1}$ larger, we can assume that for $m \geq M_{1}$ we have $f^{m}(\overline{X} - U_{-}) \subseteq N(A_+,k+4\delta)$ by Proposition~\ref{prop:ns}.  Since $gx, x \notin U_{-}$, we have that $f^{m}gx,f^{m}x \in N(A_+,k+4\delta)$ and so $\GP{f^{m}xg}{f^{m}x}_{x} \geq k+3\delta$.  Hence $\GP{g^{-1}f^{-m}x}{f^{m}gx}_{x} \leq k + 2\delta$ as:
\begin{multline*}
 \min \left\{ \GP{g^{-1}f^{-m}x}{f^{m}gx}_{x},\GP{f^{m}gx}{f^{m}x}_{x}  \right\} - \delta \leq \GP{g^{-1}f^{-m}x}{f^{m}x}_{x} \leq k + \delta.
\end{multline*}
Therefore for $C = k + 2\delta$ and $m \geq M_{1}$ we have:
\begin{align*}
d(x,f^{m}gf^{m}gx) &= d(g^{-1}f^{-m}x,gf^{m}x) \\
& \geq d(g^{-1}f^{-m}x,x) + d(x,f^{m}gx) - 2C \\
& = 2d_{m} - 2C.
\end{align*}  
The proof now proceeds exactly as that of Lemma~4.3 in~\cite{ar:CP12-1}.
\end{proof}

\begin{proof}[Proof of Theorem~\ref{th:uniform hyperbolic}] Using lemmas~\ref{lem:uniform hyperbolic 1} and \ref{lem:uniform hyperbolic 2} the proof of Theorem~\ref{th:uniform hyperbolic} proceeds exactly like that of Theorem~4.1 in~\cite{ar:CP12-1}.  We repeat the argument here.

Fix $x \in U_{+} \cap X$, and let $t >0$ and $C \geq 0$ be the constants from Lemma~\ref{lem:uniform hyperbolic 1} ,and $\epsilon > 0$ and $M_{1} \geq 0$  be the constants from Lemma~\ref{lem:uniform hyperbolic 2}.  For $m \geq M_{1}$ we set $L_{m} = d(x,f^mgx) \geq mt - C$.  As in Lemma~\ref{lem:uniform hyperbolic 2}, let $\alpha_{m} \from [0,L_{m}] \to X$ be a geodesic connecting $x$ to $f^mgx$, and let $\beta_{m} = \alpha_{m} \cdot f^mg\alpha_{m}$.  Then define a path $\gamma \from \RR \to X$ by: 
\[ \gamma = \cdots (f^mg)^{-1}\beta_{m} \bigcup_{\alpha_{m}} \beta_{m} \bigcup_{f^mg\alpha_{m}} f^mg\beta_{m} \bigcup_{(f^mg)^2\alpha_{m}} (f^mg)^2\beta_{m} \cdots \]       
See Figure~\ref{fig:concatentation}.

\begin{figure}[ht]
\begin{tikzpicture}[scale=1.75]
\tikzstyle{vertex} =[circle,draw,fill=black,thick, inner sep=0pt,minimum size= 0.75 mm]
\draw[thick] (-3.5,0.5) -- (-3,0) -- (-2,1) (-1,0) -- (0,1) -- (1,0) -- (2,1) -- (3,0) -- (3.5,0.5);
\draw[thick,red] (-2,1) -- (-1,0);
\draw[thick,blue] (-2.9,-0.1) -- (-2,0.8) -- (-1.1,-0.1);
\draw[thick,blue] (-1.9,1.1) -- (-1,0.2) -- (-0.1,1.1);
\node at (-2,0.1) {$(f^mg)^{-1}\beta_{m}$};
\node at (-1,0.7) {$\beta_{m}$};
\node[vertex,label={below:$(f^mg)^{-1}x$}] at (-3,0) {};
\node[vertex,label={below:$f^mgx$}] at (-1,0) {};
\node[vertex,label={below:$(f^mg)^{3}x$}] at (1,0) {};
\node[vertex,label={below:$(f^mg)^{5}x$}] at (3,0) {};
\node[vertex,label={above:$x$}] at (-2,1) {};
\node[vertex,label={above:$(f^mg)^{2}x$}] at (0,1) {};
\node[vertex,label={above:$(f^mg)^{4}x$}] at (2,1) {};
\end{tikzpicture}
\caption{The path $\gamma$ in the proof of Theorem~\ref{th:uniform hyperbolic}.}\label{fig:concatentation}
\end{figure}
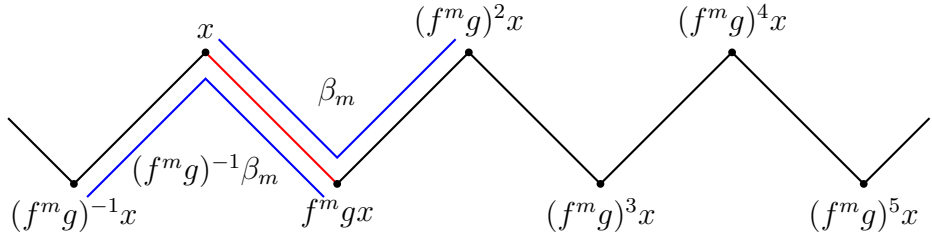

By Lemma~\ref{lem:uniform hyperbolic 2}, $\gamma$ is an $L_{m}$--local $(1,\epsilon)$--quasi-geodesic and hence for $m$ large enough, $\gamma$ is a $(\lambda',\epsilon')$--quasi-geodesic from some $\lambda' \geq 1$ and $\epsilon' \geq 0$ (see~\cite[III.H.1.7 and III.H.1.13]{bk:BH99} or \cite[Theorem~4.4]{ar:CP12-1}).  

Let $N$ be such that $t = \frac{1}{\lambda'}L_{m}N - \epsilon' > 0$.  Then for any $k \neq\ell \in \ZZ$ we have 
\begin{equation*}
d((f^mg)^{Nk}x,(f^mg)^{N\ell}x) \geq \frac{1}{\lambda'}L_{m}N\abs{k - \ell} - \epsilon' \geq t \abs{k - \ell}.
\end{equation*}
Thus $(f^{m}g)^{N}$ is hyperbolic and therefore so is $f^{m}g$.
\end{proof}

We conclude this section with an application of Theorem~\ref{th:uniform hyperbolic} to dependent hyperbolic isometries (Theorem~\cite[Theorem~4.1]{ar:CP12-1} would suffice as well).

\begin{proposition}\label{prop:dependent hyperbolic}
Suppose $X$ is a $\delta$--hyperbolic space and $f,g \in \Isom(X)$ are dependent hyperbolic isometries.  There is an $N \geq 0$ such that if $n \geq N$ then $fg^{n}$ is hyperbolic.
\end{proposition}

\begin{proof}
Let $A_{+}$, $A_{-}$, $B_{+}$, $B_{-} \in \partial X$ be the attracting and repelling fixed points for $f$ and $g$ respectively.  Then $fB_{+} \neq B_{-}$ as one of these points is fixed by $f$.  Thus there are neighborhoods $V_{+}$ and $V_{-}$ for $B_{+}$ and $B_{-}$ respectively in $\overline{X}$ such that $fV_{+} \cap V_{-} = \emptyset$.  Let $N$ be the constant from Theorem~\ref{th:uniform hyperbolic} applied to this set-up after interchanging the roles of $f$ and $g$.  Hence $g^{n}f$, and therefore the conjugate $fg^{n}$ as well, is hyperbolic when $n \geq N$.  
\end{proof}


\section{Finding neighborhoods}\label{sec:neighborhoods}

We now need to understand when we can find neighborhoods satisfying the hypotheses of Theorem~\ref{th:uniform hyperbolic} for all powers (or at least lots of powers) of a given $g$.  There are two cases that we examine: first when $g$ has a fixed point and second when $g$ is hyperbolic.

\begin{proposition}\label{prop:elliptic}
Suppose $X$ is a $\delta$--hyperbolic space and $f \in \Isom(X)$ is a hyperbolic isometry with attracting and repelling fixed points $A_{+}$ and $A_{-}$ in $\partial X$. Suppose $g \in \Isom(X)$ has a fixed point and consider a sequence of elements $(g_{k})_{k \in \NN} \subseteq  \I{g}$.  Then either:
\begin{enumerate}
\item there are disjoint neighborhoods $U_{+}$ and $U_{-}$ of $A_{+}$ and $A_{-}$ respectively and a constant $M \geq 1$ such that if $k \geq M$ then $g_{k}U_{+} \cap U_{-} = \emptyset$; or \label{item:elliptic 1}
\item there is a subsequence $(g_{k_{n}})$ so that $g_{k_{n}}A_{+} \to A_{-}$\label{item:elliptic 2}.\end{enumerate}
Further, if $gA_{-} = A_{-}$ then \eqref{item:elliptic 1} holds.
\end{proposition}

\begin{proof}
Let $p \in X$ be such that $gp = p$.  Thus $g_{k}p = p$ for all $k \in \NN$.  

Fix a system of decreasing disjoint neighborhoods $U_{-}^{k}$ of $A_{-}$ and $U_{+}^{k}$ of $A_{+}$ indexed by the natural numbers so that:
\begin{align*}
\GP{x}{A_{+}}_{p} \geq k + \delta & \mbox{ for } x \in U_{+}^{k} \mbox{, and}\\
\GP{x}{A_{-}}_{p} \geq k + \delta & \mbox{ for } x \in U_{-}^{k}. 
\end{align*}
This implies that for any two points $x,x'\in U_{+}^{k}$
we have that 
\[ \GP{x}{x'}_{p} \geq \min \{ \GP{x}{A_{+}}_{p},\GP{x'}{A_{+}}_{p} \} - \delta \geq k. \] 
Likewise for any two points $y,y' \in U_{-}^{k}$ we have that $\GP{y}{y'}_p\geq k$.

For each $n \in \NN$, define $I_n = \{ k \in \NN \mid g_kU_+^n \cap U_-^n \neq \emptyset \}$.  If $I_n$ is a finite set for some $n$, then \eqref{item:elliptic 1} holds for the neighborhoods $U_{-}=U_-^n$ and $U_{+}=U_+^n$ where $M = \max I_{n} + 1$. 

Otherwise, there is a strictly increasing sequence $(k_{n})_{n \in \NN}$ such that $k_{n} \in I_{n}$.  Hence, for each $n \in \NN$, there is an element $x_{n} \in U_{+}^{n}$ such that $g_{k_{n}}x_{n} \in U_{-}^{n}$.  In particular,
\begin{equation}\label{eq:finite then fix 1}
\GP{g_{k_{n}}x_{n}}{A_{-}}_{p}\geq n+\delta. 
\end{equation}
On the other hand, since $x_{n} \in U_{+}^{n}$ and $g_{k_{n}}$ fixes the point $p$, we have
\begin{align}\label{eq:finite then fix 2}
\GP{g_{k_{n}}x_{n}}{g_{k_{n}}A_{+}}_{p} & = \GP{g_{k_{n}}x_{n}}{g_{k_{n}}A_{+}}_{g_{k_{n}}p}\notag \\ 
&= \GP{x_{n}}{A_{+}}_p \geq n+\delta.
\end{align}
Combining \eqref{eq:finite then fix 1} and \eqref{eq:finite then fix 2}, we get $\GP{g_{k_{n}}A_{+}}{A_{-}}_{p}\geq n$ for any $n \in \NN$.  Hence \eqref{item:elliptic 2} holds.  

Now suppose that $gA_{-} = A_{-}$.  As $A_{+} \neq A_{-}$, there is a constant $D \geq 0$ such that $\GP{f^{-k}p}{f^{k}p}_{p} \leq D$ for all $k \in \NN$.  
For any $n \in \ZZ$, we have that $\GP{f^{-k}p}{g^{n}f^{-k}p}_{p} \to \infty$ as $k \to \infty$.  In particular, for each $n \in \ZZ$, there is a constant $K_{n} \geq 0$ such that $\GP{f^{-k}p}{g^{n}f^{-k}p}_{p} \geq D + \delta$ for $k \geq K_{n}$.  Therefore $\GP{g^{n}f^{-k}p}{f^{k}p}_{p} \leq D + \delta$ for $k \geq K_{n}$ as:
\[ \GP{f^{-k}p}{f^{k}p}_{p} \geq \min\left\{\GP{f^{-k}p}{g^{n}f^{-k}p}_{p},\GP{g^{n}f^{-k}p}{f^{k}p}_{p} \right\} - \delta. \]  As $gp = p$, we have $\GP{f^{-k}p}{g^{n}f^{k}p}_{p} = \GP{g^{-n}f^{-k}p}{f^{k}p}_{p}$ and so we see that $\GP{f^{-k}p}{g^{n}f^{k}p}_{p} \leq D + \delta$ for $k \geq K_{-n}$.  This shows that \eqref{item:elliptic 2} cannot hold if $gA_{-} = A_{-}$. 
\end{proof}

\begin{proposition}\label{prop:independent hyperbolic}
Suppose $X$ is a $\delta$--hyperbolic space and $f, g \in \Isom(X)$ are independent hyperbolic isometries.  There are disjoint neighborhoods $U_{+}$ and $U_{-}$ of $A_{+}$ and $A_{-}$ and an $N \geq 1$ such that if $k \geq N$ then $g^{k}U_{+} \cap U_{-} = \emptyset$.
\end{proposition}

\begin{proof}
Let $A_{+}$, $A_{-}$, $B_{+}$, $B_{-} \in \partial X$ be the attracting and repelling fixed points for $f$ and $g$ respectively.  As $f$ and $g$ are independent, the set $\{A_{-}, A_{+}, B_{-}, B_{+}\}$ consists of 4 distinct points.  Take mutually disjoint open neighborhoods $U_{-}, U_{+}, V_{-}, V_{+}$ of $A_{-}, A_{+}, B_{-}, B_{+}$ respectively. North-South dynamics of the action of $g$ on $\overline{X}$ implies that there exist a $N \geq 1$ such that $g^{k}(\overline{X} - V_{-})\subset V_{+}$ for all $k \ge N$. In particular, $g^{k}U_{+} \subseteq V_{+}$ and since $V_{+} \cap U_{-}=\emptyset$ we see that $g^{k}U_{+} \cap U_{-} = \emptyset$ for $k \geq N$.
\end{proof}


\section{Simultaneously producing hyperbolic isometries}\label{sec:simultaneous}

We can now apply the above propositions via a careful induction to prove the main result.

\begin{theorem}\label{th:constructing hyperbolic actions}
Suppose that $\{ X_{i} \}_{i = 1,\ldots, n}$ is a collection of $\delta$--hyperbolic spaces, $G$ is a group and for each $i = 1,\ldots, n$ there is a homomorphism $\rho_{i} \from G \to \Isom(X_{i})$ such that:
\begin{enumerate}
\item there is an element $f_{i} \in G$ such that $\rho_{i}(f_{i})$ is hyperbolic; and
\item for each $g \in G$, either $\rho_{i}(g)$ has a periodic orbit or is hyperbolic.
\end{enumerate}
Then there is an $f \in G$ such that $\rho_{i}(f)$ is hyperbolic for all $i = 1,\ldots,n$.
\end{theorem}

\begin{proof}
We will prove this by induction.  The case $n=1$ obviously holds by hypothesis. 

For $n \geq 2$, by induction there is an $f\in G$ such that for $i = 1,\ldots, n-1$ the isometry $\rho_{i}(f) \in \Isom (X_{i})$ is hyperbolic.   For $i = 1,\ldots,n-1$, let $A_{+}^{i}, A_{-}^{i} \in \bd X_{i}$ be the attracting and repelling fixed points of the hyperbolic isometry $\rho_{i}(f)$.  By hypothesis, there is a $g \in G$ so that $\rho_{n}(g) \in \Isom(X_{n})$ is hyperbolic.  Let $B_{+},B_{-} \in \bd X_{n}$ be the attracting and repelling fixed points of the hyperbolic isometry $\rho_{n}(g)$.  Our goal is to find $a,b \in \NN$ so that $\rho_{i}(f^{a}g^{b})$ is hyperbolic for each $i = 1,\ldots,n$.  

We begin with some simplifications.  If $\rho_{n}(f) \in \Isom(X_{n})$ is hyperbolic then there is nothing to prove, so assume that $\rho_{n}(f)$ has a periodic orbit, and so after replacing $f$ by a power we have that $f$ has a fixed point.  By replacing $g$ with a power if necessary, we can assume that for $i = 1,\ldots, n-1$ the isometry $\rho_{i}(g)$ is either the identity or has infinite order.  In fact, we can assume that $\rho_{i}(g)$ has infinite order. Indeed, if $\rho_{i}(g)$ is the identity, then for all $a,b \in \NN$ we have $\rho_{i}(f^{a}g^{b}) = \rho_{i}(f^{a})$, which is hyperbolic by the inductive hypothesis. Hence any powers for $f$ and $g$ that work for all other indices between $1$ and $n-1$ necessarily work for this index $i$ as well.  Again, by replacing $g$ with a power if necessary, we can assume that for each $i = 1, \ldots,n-1$ either $\rho_{i}(g)A_{-}^{i} = A_{-}^{i}$ or $\rho_{i}(g^{b})A_{-}^{i} \neq A_{-}^{i}$ for each $b \in \ZZ - \{0\}$.  Finally, replacing $g$ with a further power necessary, we can assume that for each $i = 1, \ldots,n-1$ if  $\rho_i(g)$ is not hyperbolic, then it has a fixed point. Analogously, by replacing $f$ with a power if necessary, we can assume that the isometry $\rho_{n}(f)$ has infinite order and that either $\rho_{n}(f)B_{-} = B_{-}$ or $\rho_{n}(f^{a})B_{-} \neq B_{-}$ for $a \in \ZZ - \{0\}$.

There are various scenarios depending on the dynamics of the isometries $\rho_{i}(g)$ and $\rho_{n}(f)$.

Let $E \subseteq \{1,\ldots,n-1\}$ be the subset where the isometries $\rho_{i}(g)$ has a fixed point. Let $H = \{1,\ldots,n-1\} - E$; this is of course the subset where $\rho_{i}(g)$ is hyperbolic.  For $i \in H$, let $B_{+}^{i}, B_{-}^{i} \in \bd X_{i}$ be the attracting and repelling fixed points of the hyperbolic isometry $\rho_{i}(g)$.  We further identify the subset $H' \subseteq H$ where $\rho_{i}(f)$ and $\rho_{i}(g)$ are independent.  

We first deal with the spaces where $\rho_{i}(g)$ is hyperbolic.  To this end, fix $i \in H$.  

If $i \in H'$, then by Proposition~\ref{prop:independent hyperbolic} there are disjoint neighborhoods $U_{+}^{i},U_{-}^{i} \subset \overline{X_{i}}$ of $A_{+}^{i}$ and $A_{-}^{i}$ respectively and an $N_{i}$ so that for $k \geq N_{i}$ we have $\rho_{i}(g^{k})U_{+}^{i} \cap U_{-}^{i} = \emptyset$.  Applying Theorem~\ref{th:uniform hyperbolic} with the neighborhoods $U_{+}$ and $U_{-}$, there is a $M_{i}$ so that for $a \geq M_{i}$ and $b \geq N_{i}$ the element $\rho_{i}(f^{a}g^{b})$ is hyperbolic.  

If $i \in H - H'$ then, by Proposition~\ref{prop:dependent hyperbolic}, for each $a \in \NN$ there is a constant $C_{i}(a) \geq 0$ such that the isometry $\rho_{i}(f^{a}g^{b})$ is hyperbolic if $b \geq C_{i}(a)$.    

To create a uniform statements in the sequel, for $i \notin H'$ (including $i \in E$), set $C_{i}(a) = 0$ for all $a \in \NN$. Also, set $M_{i} = N_{i} = 0$ for $i \in H - H'$.   

Summarizing the situation for far, we let $\mathsf{M_{0}} = \max\{ M_{i} \mid i \in H \}$ and $\mathsf{N_{0}} = \max\{ N_{i} \mid i \in H \}$.  Then, at this point, we know that if $i \in H$, $a \geq \mathsf{M_{0}}$ and $b \geq \mathsf{N_{0}}$ then the element $\rho_{i}(f^{a}g^{b})$ is hyperbolic so long as $b \geq C_{i}(a)$.

Next we deal with the spaces where $\rho_{i}(g)$ has a fixed point.  To this end, fix $i \in E$. 

Let $E' \subseteq E$ be the subset where condition~\eqref{item:elliptic 1} of Proposition~\ref{prop:elliptic} holds using $\rho_{i}(g_{k}) = \rho_{i}(g^{\mathsf{N_{0}}+k})$.  The analysis here is similar to the the case when $i \in H'$.  By assumption, for $i \in E'$, there are disjoint neighborhoods $U_{+}^{i},U_{-}^{i} \subset \overline{X_{i}}$ of $A_{+}^{i}$ and $A_{-}^{i}$ respectively and an $N_{i}$ so that for $k \geq N_{i}$ we have $\rho_{i}(g_{k})U_{+}^{i} \cap U_{-}^{i} = \emptyset$.   Applying Theorem~\ref{th:uniform hyperbolic} with the neighborhoods $U_{+}^{i}$ and $U_{-}^{i}$, there is a $M_{i}$ so that for $a \geq M_{i}$ the element $\rho_{i}(f^{a}g^{b})$ is hyperbolic if $b \geq N_{i}$.      

To summarize again, let $\mathsf{M_{1}} =  \max\{ M_{i} \mid i \in H \cup E' \}$ and $\mathsf{N_{1}} =  \max\{ N_{i} \mid i \in H \cup E' \}$.  Then at this point, if $i \in H \cup E'$, $a \geq \mathsf{M_{1}}$ and $b \geq \mathsf{N_{1}}$ then the element $\rho_{i}(f^{a}g^{b})$ if hyperbolic so long as $b \geq C_{i}(a)$. 

It remains to deal with $E - E'$; enumerate this set by $\{i_{1},\ldots,i_{\ell}\}$.  As condition~\eqref{item:elliptic 1} of Proposition~\ref{prop:elliptic} does not hold for $\rho_{i_{1}}(g_{k}) = \rho_{i_{1}}(g^{\mathsf{N_{0}} + k})$ acting on $X_{i_{1}}$, there is a subsequence $(g^{k_{n}}) \subseteq (g^{\mathsf{N_{0}}+k})$ such that $\rho_{i_{1}}(g^{k_{n}})A_{+}^{i_{1}} \to A_{-}^{i_{1}}$.  By iteratively passing to subsequences of $(g^{k_{n}})$, we can assume that for all $i \in E - E'$, either the sequence of points $(\rho_{i}(g^{k_{n}})A_{+}^{i}) \subseteq \bd X_{i}$ converges or is discrete.  

Notice that for $i \in E - E'$, the the final statement of Proposition~\ref{prop:elliptic} implies that $\rho_{i}(g)A_{-}^{i} \neq A_{-}^{i}$.  Coupling this with one of our earlier simplifications, we have that $\rho_{i}(g^{b})A_{-}^{i} \neq A_{-}^{i}$ for all $b \in \ZZ - \{0\}$.  Hence, there is a $K \in \NN$ such that for any $i \in E - E'$ the sequence $(g^{K+k_{n}})$ satisfies either: $\rho_{i}(g^{K+k_{n}})A_{+}^{i} \to p_{i} \neq A_{-}^{i}$ or $(\rho_{i}(g^{K+k_{n}})A_{+}^{i}) \subset \bd X_{i}$ is discrete.  Indeed, suppose $\rho_{i}(g^{k_{n}})A_{+}^{i} \to p_{i}$ (nothing new is being claimed in the discrete case).  If $p_{i} \notin \{ \rho_{i}(g^{k})A_{-}^{i} \}_{k \in \ZZ}$, then neither is $\rho_{i}(g^{K})p_{i}$ for any $K \in \NN$ so $\rho_{i}(g^{K+k_{n}})A_{+}^{i} \to \rho_{i}(g^{K})p_{i} \neq A_{-}^{i}$.  Else, if $p_{i} = \rho_{i}(g^{K_{i}})A_{-}^{i}$, then for $K \neq -K_{i}$ we have $\rho_{i}(g^{K+k_{n}})A_{+}^{i} \to \rho_{i}(g^{K + K_{i}})A_{-}^{i} \neq A_{-}^{i}$.  So by taking $K \in \NN$ to avoid the finitely many such $-K_{i}$ we see that the claim holds.  Without loss of generality, we can assume that $K \geq \mathsf{N_{1}}$.

Hence for each $i \in E - E'$, by Proposition~\ref{prop:elliptic}, there are disjoint neighborhoods $U_{+}^{i}, U_{-}^{i} \subset \overline{X}$ of $A_{+}^{i}$ and $A_{-}^{i}$ respectively and an $N_{i}$ so that for $n \geq N_{i}$ we have $\rho_{i}(g^{K+k_{n}})U_{+}^{i} \cap U_{-}^{i} = \emptyset$.  Applying Theorem~\ref{th:uniform hyperbolic} with the neighborhoods $U_{+}^{i}$ and $U_{-}^{i}$, there is a $M_{i}$ so that for $a \geq M_{i}$ the element $\rho_{i}(f^{a}g^{K+k_{n}})$ is hyperbolic if $n \geq N_{i}$.

Putting all of this together, let $\mathsf{M_{2}} = \max\{M_{i} \mid 1 \leq i \leq n-1 \}$ and let $\mathsf{N_{2}}= \max\{N_{i} \mid i \in E - E' \}$.  Thus for all $i = 1,\ldots,n-1$, if $a \geq \mathsf{M_{2}}$, and $n \geq \mathsf{N_{2}}$ then $\rho_{i}(f^{a}g^{K  + k_{n}})$ is hyperbolic so long as $K + k_{n} \geq C_{i}(a)$.  (Notice that $K + k_{n} \geq K \geq \mathsf{N_{1}}$ by assumption.) 

We now work with the action on the space $X_{n}$.  Interchanging the roles of $f$ and $g$ and arguing as above using Proposition~\ref{prop:elliptic} to the sequence of isometries $(\rho_{n}(f^{\ell}))$ we either obtain a subsequence $(f^{\ell_{m}}) \subseteq (f^{\ell})$ and constants $\mathsf{M_{3}}$ and $\mathsf{N_{3}}$ so that $\rho_{n}(f^{\ell_{m}}g^{b})$ is hyperbolic if $m \geq \mathsf{M_{3}}$ and $b \geq \mathsf{N_{3}}$.

Fix some $m \geq\mathsf{M_{3}}$ large enough so that $a = \ell_{m} \geq \mathsf{M_{2}}$ and let $\mathsf{C} = \max\{C_{i}(a) \mid 1 \leq i \leq n-1\}$.  Now for $n \geq \mathsf{N_{2}}$ large enough so that $b = K + k_{n} \geq \max\{ \mathsf{C},\mathsf{N_{3}} \}$ we have that $\rho_{i}(f^{a}g^{b})$ is hyperbolic for $i=1,\ldots,n$ as desired.
\end{proof} 


\section{Application to $\Out(F_{N})$}\label{sec:application}

Let $F_N$ be a free group of rank $N\ge2$. A \emph{free factor system} of $F_{N}$ is a finite collection $\calA=\{[A_{1}],[A_{2}],\ldots,[A_{K}]  \}$ of conjugacy classes of subgroups of $F_N$, such that there exist a free factorization 
\[
F_N=A_1\ast\cdots\ast A_{K}\ast B
\]
where $B$ is a (possibly trivial) subgroup, called a \emph{cofactor}.  There is a natural partial ordering among the free factor systems: $\calA\sqsubseteq\calB$ if for each $[A] \in \calA$ there is a $[B] \in \calB$ such that $gAg^{-1} < B$ for some $g\in F_N$.  In this case, we say that $\calA$ is \emph{contained} in $\calB$ or $\calB$ is an \emph{extension} of $\calA$.  

Recall, the \emph{reduced rank} of a subgroup $A < F_{N}$ is defined as \[\rerank(A) = \min\{0,\rank(A) - 1\}.\]  We extend this to a free factor systems by addition: \[\rerank(\calA) = \sum_{k=1}^{K} \rerank(A_{k})\] where $\calA=\{[A_{1}],[A_{2}],\ldots,[A_{K}] \}$.  An extension $\calA \sqsubseteq
\calB$ is called a \emph{multi-edge extension} if $\rerank(\calB) \geq \rerank(\calA) + 2$.

The group $\Out(F_{N})$ naturally acts on the set of free factor systems as follows.  Given $\calA=\{[A_{1}],[A_{2}],\ldots,[A_{K}]  \}$, and $\varphi\in \Out(F_N)$ choose a representative $\Phi\in\Aut(F_N)$ of  $\varphi$, a realization $F_N=A_1\ast\cdots\ast A_K\ast B$ of $\calA$ and define $\varphi (\calA)$ to be the free factor system $\{[\Phi(A_1)],\ldots, [\Phi(A_K)]\}$.  Given a free factor system $\calA$ consider the subgroup $\Out(F_N;\calA)$ of $\Out(F_N)$ that stabilizes the free factor system $\calA$. The group $\Out(F_N;\calA)$ is called the \emph{outer automorphism group of $F_N$ relative to $\calA$}, or the \emph{relative outer automorphism group} if the free factor system $\calA$ is clear from context.  If $\calA = \{[A]\}$, there is a well-defined restriction homomorphism $\Out(F_{N};\calA) \to \Out(A)$ we denote by $\varphi \mapsto \varphi\mid_{A}$~\cite[Fact~1.4]{HMpart1}.   

For a subgroup $\calH < \Out(F_N)$ and $\calH$--invariant free factor systems $\calF_1\sqsubseteq\calF_2$, we say that $\calH$ is \emph{irreducible with respect to the extension $\calF_1\sqsubseteq\calF_2$} if for any $\calH$--invariant free factor system $\calF$ such that $\calF_1\sqsubseteq \calF \sqsubseteq \calF_2$ it follows that either $\calF=\calF_1$ or $\calF=\calF_2$.  We sometimes say that $\calH$ is \emph{relatively irreducible} if the extension is clear from the context.  The subgroup $\calH$ is \emph{relatively fully irreducible} if each finite index subgroup $\calH' < \calH$ is relatively irreducible.  For an individual element $\varphi\in \Out(F_{N})$, we say that $\varphi$ is relatively (fully) irreducible if the cyclic subgroup $\langle\varphi\rangle$ is relatively (fully) irreducible. 

In close analogy with Ivanov's classification of subgroups of mapping class groups \cite{Iva}, in a series of papers Handel and Mosher gave a classification of finitely generated subgroups of $\Out(F_N)~$\cite{HMIntro, HMpart1, HMpart2, HMpart3, HMpart4}.

\begin{theorem}[{\cite[Theorem~D]{HMIntro}}]\label{HMMain} For each finitely generated subgroup $\calH < \IA_{N}(\ZZ/3) < \Out(F_N)$, each maximal $\calH$--invariant filtration by free factor systems $\emptyset = \calF_{0} \sqsubset \calF_{1} \sqsubset \cdots \sqsubset \calF_{m} = \{[ F_{N}] \}$,  and each $i = 1,...,m$ such that $\calF_{i-1}\sqsubset\calF_{i}$ is a multi-edge extension, there exists $\varphi\in\calH$ which is irreducible with respect to $\calF_{i-1} \sqsubset \calF_{i}$.
\end{theorem}

Here, $\IA_{N}(\ZZ/3)$ is the finite index subgroup of $\Out(F_N)$ which is the kernel of the natural surjection \[
p \from \Out(F_N)\to\ H^{1}(F_N,\mathbb{Z}/3)\cong GL(N,\mathbb{Z}/3).
\] For elements in $\IA_{N}(\ZZ/3)$, irreducibility is equivalent to full irreducibility hence in the above statement we can also conclude that $\varphi$ is fully irreducible~\cite[Theorem~B]{HMIntro}.

Handel and Mosher conjecture that there is a single $\varphi\in \calH$ which is (fully) irreducible for each multi-edge extension $\calF_{i-1} \sqsubset \calF_{i}$~\cite[Remark following Theorem~D]{HMIntro}. The goal of this section is to prove this conjecture.  Invoking theorems of Handel--Mosher and Horbez--Guirardel, this is (essentially) an immediate application of Theorem~\ref{th:constructing hyperbolic actions}.  We state the set-up and their theorems now.   

\begin{definition}\label{def:rff}
Let $\calA$ be a free factor system of $F_{N}$.  The \emph{complex of free factor systems of $F_{N}$ relative to $\calA$}, denoted $\fF(F_{N};\calA)$, is the geometric realization of the partial ordering $\sqsubseteq$ restricted to proper free factor systems that properly contain $\calA$. 
\end{definition}

If $\calA = \{[A_{1}],[A_{2}],\ldots,[A_{K}]  \}$ is a free factor system for $F_{N}$, its \emph{depth} is defined as:
\[ {\rm D}_{\fF}(\calA) = (2N-1) - \sum_{k=1}^{K} \bigl(2\rank(A_{k}) -1\bigr) \]
The free factor system $\calA$ is \emph{nonexceptional} if ${\rm D}_{\fF}(\calA) \geq 3$.  

\begin{theorem}[{\cite[Theorem~1.2]{un:HM-relative}}]\label{hyperbolicity}
For any nonexceptional free factor system $\calA$ of $F_{N}$, the complex $\fF(F_{N};\calA)$ is positive dimensional, connected and $\delta$--hyperbolic.
\end{theorem}

Although the group $\Out(F_N)$ does not act on $\fF(F_{N};\calA)$, the natural subgroup $\Out(F_N;\calA)$ associated to the free factor system $\calA$ acts on $\fF(F_{N};\calA)$ by simplicial isometries. In a companion paper Handel and Mosher characterize the elements of $\Out(F_N;\calA)$ that act as a hyperbolic isometry of $\fF(F_{N};\calA)$:

\begin{theorem}[\cite{un:HM-relative2}]\label{loxod} For any nonexceptional free factor system $\calA$ of $F_N$, $\varphi\in\Out(F_N;\calA)$ acts as a hyperbolic isometry on $\fF(F_{N};\calA)$ if and only if $\varphi$ is fully irreducible with respect to $\calA \sqsubset \{[F_{N}]\}$. 
\end{theorem}

\begin{remark} \label{mainremark} An alternative proof of Theorem \ref{loxod} is given by Guirardel and Horbez in \cite{un:GH} using the description of the boundary of the relative free factor complex.  Further, with a slight modification of the definition of the relative free factor complex, both Handel and Mosher and Guirardel and Horbez can additionally prove that the theorem holds for the only remaining multi-edge configuration which is when $\calA=\{[A_1],[A_2],[A_3]\}$ and $F_N=A_1\ast A_2\ast A_3$. Yet another proof of Theorem \ref{loxod} when the cofactor is non-trivial is given by Radhika Gupta in \cite{un:Gupta} using dynamics on relative outer space and relative currents. 
\end{remark}

We are now ready to prove our application:

\begin{theorem}\label{th:application}
For each finitely generated subgroup $\calH < \IA_{N}(\ZZ/3)<\Out(F_N)$ and each maximal $\calH$--invariant filtration by free factor systems $\emptyset = \calF_{0} \sqsubset \calF_{1} \sqsubset \cdots \sqsubset \calF_{m} = \{[ F_{N}] \}$, there is an element $\varphi \in \calH$ such that for each $i = 1,\ldots,m$ such that $\calF_{i-1} \sqsubset \calF_{i}$ is a multi-edge extension, $\varphi$ is irreducible with respect to $\calF_{i-1} \sqsubset \calF_{i}$.
\end{theorem}

\begin{proof}  
Let $I$ be the subset of indices $i$ such that $\calF_{i-1} \sqsubset \calF_{i}$ is a multi-edge extension.

Given $i \in I$, since $\calH < \IA_{N}(\ZZ/3)$, each component of $\calF_{i-1}$ and $\calF_i$ is $\calH$--invariant~\cite[Lemma~4.2]{HMpart2}.  Moreover, by the argument at the beginning of Section~2.1 in~\cite{HMpart4}, since $\calH$ is irreducible with respect to $\calF_{i-1} \sqsubset \calF_{i}$ (this follows from maximality of the filtration) there is precisely one component $[B_{i}] \in \calF_{i}$ that is not a component of $\calF_{i-1}$.  Let $\widehat{\calA}_{i}$ be the maximal subset of $\calF_{i-1}$ such that $\widehat{\calA}_{i} \sqsubset \{[B_{i}]\}$.  Notice that this extension is again multi-edge, indeed $\rerank(B_{i}) - \rerank(\widehat{\calA}_{i}) = \rerank(\calF_{i}) - \rerank(\calF_{i-1})$.  The system $\widehat{\calA}_{i}$ can be represented by $\{ [A_{i,1}],\ldots, [A_{i,K_{i}}] \}$ where $A_{i,k} < B_{i}$ for each $k$.  Let $\calA_{i}$ be the free factor system in the subgroup $B_{i}$ consisting of the conjugacy classes in $B_{i}$ of the subgroups $A_{i,k}$.  Then a given $\varphi \in \calH$ is irreducible with respect to $\widehat{\calA}_{i} \sqsubset \{[B_{i}]\}$, equivalently $\calF_{i-1} \sqsubset \calF_{i}$ as the remaining components are the same, if and only if the restriction $\varphi\mid_{B_{i}} \in \Out(B_{i};\calA_{i})$ is irreducible relative to $\calA_{i}$.

For $i \in I$, let $X_{i} = \fF(B_{i};\calA_{i})$ and consider the action homomorphism $\rho_{i} \from \calH \to \Isom(X_{i})$ defined by $\rho_{i}(\varphi) = \varphi\mid_{B_{i}}$.  These spaces are $\delta$--hyperbolic for some $\delta$ by Theorem~\ref{hyperbolicity} and by the above discussion and Theorem~\ref{loxod}, $\rho_{i}(\varphi)$ is a hyperbolic isometry if $\varphi \in \calH$ is irreducible with respect to $\calF_{i-1} \sqsubset \calF_{i}$.  If $\rho_{i}(\varphi)$ is not irreducible with respect to $\calF_{i-1} \sqsubset \calF_{i}$, then $\rho_{i}(\varphi)$ fixes a point in $X_{i}$.  By Theorem~\ref{HMMain}, for each $i \in I$, there exist some $\varphi_{i} \in\calH$ that is irreducible with respect to $\calF_{i-1} \sqsubset \calF_{i}$ and hence $\rho_{i}(\varphi_{i})$ is a hyperbolic isometry.  

We are now in the model situation of Theorem~\ref{th:constructing hyperbolic actions}.  We conclude that there is a $\varphi \in \calH$ such that $\rho_{i}(\varphi)$ is a hyperbolic isometry for all $i \in I$.  By the above discussion, this means that $\varphi$ is (fully) irreducible with respect to $\calF_{i-1} \sqsubset \calF_{i}$ for each $i \in I$ as desired.
\end{proof}


\bibliography{bib}

\begin{thebibliography}{10}

\bibitem{col:AlonsoEtAl91}
{\sc Alonso, J.~M., and et~al.}
\newblock Notes on word hyperbolic groups.
\newblock In {\em Group theory from a geometrical viewpoint ({T}rieste, 1990)}.
  World Sci. Publ., River Edge, NJ, 1991, pp.~3--63.
\newblock Edited by H. Short.

\bibitem{BH10}
{\sc Bj{\"o}rklund, M., and Hartnick, T.}
\newblock Biharmonic functions on groups and limit theorems for quasimorphisms
  along random walks.
\newblock {\em Geom. Topol. 15}, 1 (2011), 123--143.

\bibitem{bk:BH99}
{\sc Bridson, M.~R., and Haefliger, A.}
\newblock {\em Metric spaces of non-positive curvature}, vol.~319 of {\em
  Grundlehren der Mathematischen Wissenschaften [Fundamental Principles of
  Mathematical Sciences]}.
\newblock Springer-Verlag, Berlin, 1999.

\bibitem{ar:CLM12}
{\sc Clay, M., Leininger, C., and Mangahas, J.}
\newblock The geometry of right-angled {A}rtin subgroups of mapping class
  groups.
\newblock {\em Groups Geom. Dyn. 6}, 2 (2012), 249--278.

\bibitem{ar:CP12-1}
{\sc Clay, M., and Pettet, A.}
\newblock Current twisting and nonsingular matrices.
\newblock {\em Comment. Math. Helv. 87}, 2 (2012), 385--407.

\bibitem{DTHyp}
{\sc Dowdall, S., and Taylor, S.~J.}
\newblock Hyperbolic extensions of free groups.
\newblock {P}reprint,
  \href{http://arxiv.org/abs/1406.2567}{arXiv:math/1406.2567}.

\bibitem{ar:Fujiwara15}
{\sc Fujiwara, K.}
\newblock Subgroups generated by two pseudo-{A}nosov elements in a mapping
  class group. {II}. {U}niform bound on exponents.
\newblock {\em Trans. Amer. Math. Soc. 367}, 6 (2015), 4377--4405.

\bibitem{col:Gromov87}
{\sc Gromov, M.}
\newblock Hyperbolic groups.
\newblock In {\em Essays in group theory}, vol.~8 of {\em Math. Sci. Res. Inst.
  Publ.} Springer, New York, 1987, pp.~75--263.

\bibitem{un:GH}
{\sc Guirardel, V., and Horbez, C.}
\newblock {I}n preparation, 2016.

\bibitem{un:Gultepe}
{\sc Gultepe, F.}
\newblock Fully irreducible automorphisms of the free group via {D}ehn twisting
  in {$\sharp_k(S^2 \times S^1)$}.
\newblock {P}reprint,
  \href{http://arxiv.org/abs/1411.7668}{arXiv:math/1411.7668}.

\bibitem{un:Gupta}
{\sc Gupta, R.}
\newblock {L}oxodromic elements for the relative free factor complex.
\newblock preprint, 2016.

\bibitem{HMIntro}
{\sc Handel, M., and Mosher, L.}
\newblock Subgroup decomposition in {${\rm Out}(F_n)$}: {I}ntroduction and
  research announcement.
\newblock {P}reprint,
  \href{http://arxiv.org/abs/1302.2681}{arXiv:math/1302.2681}.

\bibitem{HMpart1}
{\sc Handel, M., and Mosher, L.}
\newblock Subgroup decomposition in {${\rm Out}(F_n)$}, {P}art {I}: {G}eometric
  models.
\newblock {P}reprint,
  \href{http://arxiv.org/abs/1302.2378}{arXiv:math/1302.2378}.

\bibitem{HMpart2}
{\sc Handel, M., and Mosher, L.}
\newblock Subgroup decomposition in {${\rm Out}(F_n)$}, {P}art {II}: {A}
  relative {K}olchin theorem.
\newblock {P}reprint,
  \href{http://arxiv.org/abs/1302.2379}{arXiv:math/1302.2379}.

\bibitem{HMpart3}
{\sc Handel, M., and Mosher, L.}
\newblock Subgroup decomposition in {${\rm Out}(F_n)$}, {P}art {III}: {W}eak
  attraction theory.
\newblock {P}reprint,
  \href{http://arxiv.org/abs/1302.4712}{arXiv:math/1302.4712}.

\bibitem{HMpart4}
{\sc Handel, M., and Mosher, L.}
\newblock Subgroup decomposition in {${\rm Out}(F_n)$}, {P}art {IV}: Relatively
  irreducible subgroups.
\newblock {P}reprint,
  \href{http://arxiv.org/abs/1302.4711}{arXiv:math/1302.4711}.

\bibitem{un:HM-relative}
{\sc Handel, M., and Mosher, L.}
\newblock Relative free splitting and free factor complexes {I}: Hyperbolicity.
\newblock {P}reprint,
  \href{http://arxiv.org/abs/1407.3508}{arXiv:math/1407.3508}, 2014.

\bibitem{un:HM-relative2}
{\sc Handel, M., and Mosher, L.}
\newblock Relative free splitting and free factor complexes {II}: Loxodromic
  outer automorphisms.
\newblock {I}n preparation, 2016.

\bibitem{Hshort}
{\sc Horbez, C.}
\newblock A short proof of {H}andel and {M}osher's alternative for subgroups of
  {${\rm Out}(F_N)$}.
\newblock {\em Groups, Geometry, and Dynamics 10}, 2 (2016), 709--721.

\bibitem{Iva}
{\sc Ivanov, N.~V.}
\newblock {\em Subgroups of {T}eichm\"uller modular groups}, vol.~115 of {\em
  Translations of Mathematical Monographs}.
\newblock American Mathematical Society, Providence, RI, 1992.
\newblock Translated from the Russian by E. J. F. Primrose and revised by the
  author.

\bibitem{col:KB02}
{\sc Kapovich, I., and Benakli, N.}
\newblock Boundaries of hyperbolic groups.
\newblock In {\em Combinatorial and geometric group theory ({N}ew {Y}ork,
  2000/{H}oboken, {NJ}, 2001)}, vol.~296 of {\em Contemp. Math.} Amer. Math.
  Soc., Providence, RI, 2002, pp.~39--93.

\bibitem{MT}
{\sc Maher, J., and Tiozzo, G.}
\newblock Random walks on weakly hyperbolic groups.
\newblock 2014.

\bibitem{ar:Mangahas13}
{\sc Mangahas, J.}
\newblock A recipe for short-word pseudo-{A}nosovs.
\newblock {\em Amer. J. Math. 135}, 4 (2013), 1087--1116.

\bibitem{ar:Taylor14}
{\sc Taylor, S.~J.}
\newblock A note on subfactor projections.
\newblock {\em Algebr. Geom. Topol. 14}, 2 (2014), 805--821.

\end{thebibliography}
\bibliographystyle{acm}

\end{document}